\documentclass[a4paper,11pt]{amsart}
\usepackage{amssymb, amsthm, bm, comment, enumerate, hyperref, mathtools, xspace}
\usepackage[paper=a4paper,                     
            left=27mm,                         
            right=27mm,                        
            top=25mm,                          
            bottom=25mm,                       
            bindingoffset=0mm                  
           ]{geometry}

\theoremstyle{plain}
\newtheorem{theorem}{Theorem}[section]
\newtheorem{conjecture}[theorem]{Conjecture}
\newtheorem{corollary}[theorem]{Corollary}
\newtheorem{lemma}[theorem]{Lemma}
\newtheorem{problem}[theorem]{Problem}
\newtheorem{proposition}[theorem]{Proposition}

\theoremstyle{definition}

\theoremstyle{remark}
\newtheorem{example}[theorem]{Example}

\newtheorem{remark}[theorem]{Remark}

\DeclareMathOperator{\ZZ}{\mathbb{Z}}

\DeclareMathOperator{\Hilb}{Hilb}

\DeclareMathOperator{\Proj}{Proj}
\DeclareMathOperator{\rad}{rad}
\DeclareMathOperator{\toset}{set}

\makeatletter
\def\@tocline#1#2#3#4#5#6#7{\relax
  \ifnum #1>\c@tocdepth
  \else
    \par \addpenalty\@secpenalty\addvspace{#2}
    \begingroup \hyphenpenalty\@M
    \@ifempty{#4}{
      \@tempdima\csname r@tocindent\number#1\endcsname\relax
    }{
      \@tempdima#4\relax
    }
    \parindent\z@ \leftskip#3\relax \advance\leftskip\@tempdima\relax
    \rightskip\@pnumwidth plus4em \parfillskip-\@pnumwidth
    #5\leavevmode\hskip-\@tempdima
      \ifcase #1
       \or\or \hskip 1em \or \hskip 2em \else \hskip 3em \fi
      #6\nobreak\relax
    \dotfill\hbox to\@pnumwidth{\@tocpagenum{#7}}\par
    \nobreak
    \endgroup
  \fi}
\makeatother
\title[CGFs, empty WCIs and positivity]{Cyclotomic generating functions, empty weighted complete intersections and positivity}
\author{Mona Gatzweiler}
\address[M. Gatzweiler]{Fakultät für Mathematik, Universität Wien, Österreich}
\email{mona.gatzweiler@univie.ac.at}
\author{Fabi\'an Levic\'an-Santib\'a\~nez}
\address[F. Levic\'an-Santib\'a\~nez]{Fakultät für Mathematik, Universität Wien, Österreich}
\email{fabian.levican@univie.ac.at}
\author{Atsuro Yoshida}
\address[A. Yoshida]{Fakultät für Mathematik, Universität Wien, Österreich}
\email{atsuro.yoshida@univie.ac.at}
\date{\today}

\begin{document}

\begin{abstract}
We give a sufficient combinatorial condition for the non-negativity of the coefficients of polynomial quotients of products of $q$-integers, also known as \textit{cyclotomic generating functions (CGFs)}. This slightly extends work by Iano-Fletcher, Pizzato, Sano and Tasin, who studied this condition as a criterion for quasismoothness of complete intersections in weighted projective spaces. As a consequence, we solve a problem by Billey and Swanson, prove most cases of an unpublished conjecture by Stanton and most cases of two conjectures by Gatzweiler and Krattenthaler. We also study sufficient conditions given by structural properties of the division lattice.
\end{abstract}

\maketitle

\tableofcontents

\section{Introduction}
Let $f(q)$ be a non-zero polynomial with non-negative integer coefficients in the variable~$q$. Then, the polynomial $f(q)$ is called a \emph{cyclotomic generating function (CGF)} if there exist $\alpha, \beta \in \ZZ_{\geq 0}$ and $a_1, \ldots, a_n, b_1, \ldots, b_n \in \mathbb{Z}_+$ such that
\begin{equation}
    f(q) = \alpha q^{\beta} \prod_{i=1}^n \frac{[a_i]_q}{[b_i]_q} = \alpha q^{\beta} \prod_{i=1}^n \frac{1-q^{a_i}}{1-q^{b_i}}, \label{equation:CGF-def}
\end{equation}
where $[n]_q = (1-q^n)/(1-q)$ is a $q$-integer.
CGFs are of great significance in combinatorics as they are the generating functions of many combinatorial objects with respect to statistics on these objects. A prototypical example of such generating functions is the $q$-binomial coefficient, being the generating function of integer partitions inside a $k \times (n-k)$ frame with respect to the area statistic:
\begin{equation}
    \sum_{\lambda \subset k \times (n-l)} q^{|\lambda|} =\binom{n}{k}_q =  \frac{\prod_{i=1}^n (1-q^i)}{\prod_{i=1}^k (1-q^l) \prod_{i=1}^{n-k}(1-q^i)}. \label{equation:q-binom}
\end{equation}
For further reading, the paper \cite{billey-swanson-cyclotomic} offers an excellent survey on CGFs from different perspectives and a long (but not exhaustive) list of sources of CGFs throughout combinatorics.

Considering the expression in \eqref{equation:q-binom}, it is not \emph{a priori} clear that the right-hand side is a polynomial with non-negative integer coefficients.
In general, given $a_1, \ldots, a_n, b_1, \ldots, b_n \in \mathbb{Z}_+$ and assuming that the quotient
\begin{equation}
    \frac{\prod_{i=1}^n{1-q^{a_i}}}{\prod_{i=1}^n{1-q^{b_i}}}. \label{equation:A-B-fraction}
\end{equation}
simplifies into a polynomial in $q$, it is not easy to determine whether it has non-negative integer coefficients. Nevertheless, there are many known rational functions in the form of \eqref{equation:A-B-fraction} that are polynomials with non-negative integer coefficients, as all CGFs are essentially of this form up to a constant and a power of $q$.

The polynomiality and non-negativity of fractions of the form  \eqref{equation:A-B-fraction} have been extensively studied in the literature. For instance, in the unpublished note \cite{stanton-fake}, Stanton posed the following conjecture related to \emph{fake Gaussian sequences}.

\begin{conjecture} \cite[Conjecture 1]{stanton-fake} \label{conjecture:conjecture-stanton}
    Let $\mathbf{a}=(a_1, \ldots, a_n)$ be a sequence of non-negative integers.
    If
    \begin{equation}
        \frac{\prod_{i=1}^n (1-q^{m+i})^{a_i}}{\prod_{i=1}^n (1-q^i)^{a_i}}. \label{equation:stanton-fraction}
    \end{equation}
    is a polynomial for all non-negative integers~$m$ and $\mathbf{a}$ is symmetric, then \eqref{equation:stanton-fraction} has non-negative integer coefficients for all $m$.
\end{conjecture}
More recently, Gatzweiler and Krattenthaler \cite{gatzweiler-krattenthaler-positivity} study the following conjecture.
\begin{conjecture} \cite[Conjecture 1]{gatzweiler-krattenthaler-positivity} \label{conjecture:conjecture-gk}
    Let $n, k, l \in \ZZ_{\geq 0}$. If
    \begin{equation}
        \frac{\binom{n}{k}_q}{\binom{n}{l}_q} = \frac{[l]_q![n-l]_q!}{[k]_q![n-k]_q!} = \frac{\prod_{i=1}^l (1-q^i) \prod_{i=1}^{n-l} (1-q^i)}{\prod_{i=1}^k (1-q^i) \prod_{i=1}^{n-k} (1-q^i)} \label{equation:GK}
    \end{equation}
    is a polynomial, then it has non-negative integer coefficients.
\end{conjecture}
Note that in Conjecture \ref{conjecture:conjecture-gk} it suffices to consider the case where $1 \leq l < k  \leq n/2$. In this case, \eqref{equation:GK} is equal to
\begin{equation}
    \frac{\prod_{i=n-k+1}^{n-l} (1-q^i)}{\prod_{i=l+1}^{k} (1-q^i)}.
\end{equation}
In \cite{gatzweiler-krattenthaler-positivity}, the authors prove Conjecture~\ref{conjecture:conjecture-gk} for the cases in which $l = k, k-1, k-2, k-3$ and $l = 0, 1, 2$ by direct arguments involving cyclotomic polynomials. 

Furthermore, Gatzweiler and Krattenthaler claim that, in Conjecture~\ref{conjecture:conjecture-stanton}, the conditions regarding the polynomiality of $m$ and the symmetry of $\mathbf{a}$ can be dropped. In other words:

\begin{conjecture} \cite[Conjecture 4]{gatzweiler-krattenthaler-positivity} \label{conjecture:conjecture-stanton-gk}
    Let $m, n$ be positive integers with $m \geq n$, and let $\mathbf{a}$ be a sequence of non-negative integers.
    If \eqref{equation:stanton-fraction} is a polynomial in $q$, then it has non-negative coefficients.
\end{conjecture}

In this article, we study sufficient conditions for rational functions of the form \eqref{equation:A-B-fraction} to have non-negative integer coefficients in the cases where these functions are polynomials. We classify our criteria into two categories: algebro-geometric and lattice-theoretic. The algebro-geometric point of view is particularly effective; we slightly extend ideas by \cite{iano-fletcher-working} and \cite{pizzato-sano-tasin-effective} to give a necessary and sufficient combinatorial condition for the existence of \emph{empty} complete intersections in weighted projective spaces. This condition involves a family of Frobenius coin problems. It can be stated as follows:

\begin{proposition}[Proposition \ref{proposition:regular-sequence-iff-wci-iff-frobenius-coin-problems}] \label{proposition:regular-sequence-iff-wci-iff-frobenius-coin-problems-in-intro}
    Let $a_1, \ldots, a_n, b_1, \ldots, b_n \in \mathbb{Z}_+$. Let $A = [a_i]_{i = 1, \ldots, n}$ and $B = [b_i]_{i = 1, \ldots, n}$ be the corresponding multisets. Then, the following are equivalent:
    \begin{enumerate}[1.]
        \item There exists a regular sequence of homogeneous elements $f_1, \ldots, f_n$ of maximal length with $\deg(f_i) = a_i$ in $\mathbb{C}[x_1, \ldots, x_n]$ with $\deg(x_i) = b_i$. In other words, there exists an empty weighted complete intersection $X_{a_1, \ldots, a_n} \subset \mathbb{P}(b_1, \ldots, b_n)$. 
        \item For every subset of indices $I \subseteq \{1, \ldots, n\}$, the inequality $\# (A~\overline{\cap}~\langle \toset(B_I) \rangle_{\text{semigroup}}) \geq \# I$ holds.
    \end{enumerate}
\end{proposition}

The reader is referred to Sections \ref{section:preliminaries} and \ref{section:algebraic-geometry} for definitions and notation.

In \cite{billey-swanson-cyclotomic}, the authors define the \emph{HSOP monoid} (HSOP stands for \emph{homogeneous system of parameters}) as the set of CGFs that are Hilbert series of the quotient of $\mathbb{C}[x_1, \ldots, x_n]$ by a regular sequence of homogeneous elements $f_1, \ldots, f_n$ of maximal length. They note that this set is a monoid. They propose the following:
\begin{problem} \cite[Problem 60]{billey-swanson-cyclotomic} \label{problem:problem-bs}
    Besides identifying a specific homogeneous system of parameters in a graded ring, how can one test membership in the HSOP monoid?
\end{problem}

Proposition~\ref{proposition:regular-sequence-iff-wci-iff-frobenius-coin-problems-in-intro} (Proposition~\ref{proposition:regular-sequence-iff-wci-iff-frobenius-coin-problems}) gives a solution to Problem \ref{problem:problem-bs}. The following is our main theorem, which provides a more practical sufficient condition.

\begin{theorem}[Theorem~\ref{theorem:main-theorem}]
    Let $a_1, \ldots, a_n, b_1, \ldots, b_n \in \mathbb{Z}_+$ with $a_i \leq a_{i + 1}$ and $b_i \leq b_{i + 1}$ for all $i = 1, \ldots, n - 1$. Let $A = [a_i]_{i = 1, \ldots, n} $ and $ B = [b_i]_{i = 1, \ldots, n}$ be the corresponding multisets. Assume that
    \[
    \prod_{i = 1}^n \frac{1 - q^{a_i}}{1 - q^{b_i}} \in \mathbb{Z}[q].
    \]
    Assume further that
    \[
    2b_n\left \lfloor \frac{b_{n - 1}}{2} \right \rfloor - b_1 < a_1.
    \]
    Then, for every subset of indices $I \subseteq \{1, \ldots, n\}$, $\# (A~\overline{\cap}~\langle \toset(B_I) \rangle_{\text{semigroup}}) \geq \# I$. In particular,
    \[
    \prod_{i = 1}^n \frac{1 - q^{a_i}}{1 - q^{b_i}} \in \mathbb{Z}_{\geq 0}[q].
    \]
\end{theorem}

Applying this theorem to Conjecture~\ref{conjecture:conjecture-gk}, we obtain the following corollary.
\begin{corollary}[Corollary~\ref{corollary:conjecture-gk-upper-bound}]
    Conjecture \ref{conjecture:conjecture-gk} is true if
    \[
    2k \left \lfloor \frac{k - 1}{2} \right \rfloor + k - l < n.
    \]
    In particular, it is true if $n \geq k^2$. Hence, it is true except perhaps for finitely many choices of $l, n$ for each choice of $k$.
\end{corollary}
Furthermore, applying Theorem~\ref{theorem:main-theorem} to Conjecture~\ref{conjecture:conjecture-stanton-gk}, we obtain the following corollary.
\begin{corollary}[Corollary~\ref{corollary:conjecture-stanton-upper-bound}]
    Conjectures \ref{conjecture:conjecture-stanton} and \ref{conjecture:conjecture-stanton-gk} are true if
    \[
    2n \left \lfloor \frac{n}{2} \right \rfloor - 2 < m.
    \]
    In particular, they are true if $m \geq n^2$. 
    Hence, they are true except perhaps for finitely many choices of $m$ for each choice of $n$.
\end{corollary}

From the lattice-theoretic point of view, we give sufficient criteria involving lower order ideals of the division lattice of positive integers (Proposition~\ref{proposition:quotient-ai-bi}). However, we also note that lattice-theoretic conditions cannot be necessary for a rational function of the form \eqref{equation:A-B-fraction} that is a polynomial to have non-negative coefficients. This is because certain CGFs are non-negative for \emph{analytic} reasons; we mention relations to Poincaré multipliers and normal distributions.

\section{Preliminaries}\label{section:preliminaries}
In this section we summarise preliminaries that will be used in the rest of the paper. We recall basic facts about cyclotomic polynomials, CGFs, the Frobenius coin problem and regular sequences. We also introduce notation for multisets.

The symbols $\mathbb{Z}_+$ and $\mathbb{Z}_{\geq 0}$ will always denote the positive and non-negative integers, respectively.

\subsection{Multisets}
We use square brackets $[$ and $]$ for multisets (e.g., $[1, 1, 2, 3]$).

For a multiset $S$, we write $\toset(S)$ for its underlying set, so $\toset([1, 1, 2, 3]) = \{1, 2, 3\}$. We also write $\# S$ for its \emph{multiset cardinality}, so $\# [1, 1, 2, 3] = 4$.

For an indexed multiset $S = [s_1, \ldots, s_n]$ and a subset of indices $I \subseteq \{1, \ldots, n\}$, we define $S_I \coloneq [s_i: i \in I]$.

For two multisets $S_1$ and $S_2$, we write $S_1 \uplus S_2$ or $\biguplus_{i = 1}^2 S_i $ for \emph{multiset} or \emph{additive union} (so $[1, 1, 2, 3] = [1, 2] \uplus [1, 3]$) and $S_1 \setminus S_2$ for \emph{multiset difference} (so $[1, 3] = [1, 1, 2, 3] \setminus [1, 2, 4]$).

Finally, for a multiset $S$ and a set $S'$, we define $S~\overline{\cap}~S' \coloneq [s \in S: s \in S']$. For example, $[1, 1, 2] = [1, 1, 2, 3]~\overline{\cap}~\{1, 2\}$.

\subsection{Cyclotomic polynomials}
For $n \in \ZZ_{+}$, define $D(n) \coloneq \{d: d \mid n\}$, $P(n) = \{ p \text{ prime}:p \mid n \}$ and $\omega(n) \coloneq \# P(n)$. We write $\rad(n)$ for the radical of $n$.

Let $n \in \ZZ_{+}$. The \emph{cyclotomic polynomial~$\Phi_n(q)$ of order $n$} is defined by
\begin{equation*}
    \Phi_n(q) \coloneq \prod_{\substack{1 \leq k \leq n \\ \gcd(n, k)=1}} (q-e^{2 \pi i k / n}).
\end{equation*}
For $f(q) = a_0 + a_1x + \cdots + a_nx^n \in \mathbb{Z}[q]$, define $A(n) = \max_{i = 1, \ldots, n} |a_i|$. The polynomial $f(q)$  is called \emph{flat} if $|A(n)| \leq 1$. Furthermore, it is well-known that 
\[
1-q^{n} = \prod_{d|n} \Phi_d(q).
\]
We are also going to use the following standard properties of cyclotomic polynomials.
\begin{proposition}\label{proposition:cyclotomic-polynomials}
    \begin{enumerate}[1.]
        \item $\Phi_n(q)$ is flat if $n$ has at most two distinct odd prime factors \cite{migotti-theorie}. In particular, it is flat if $\omega(n) \leq 2$.
        \item $\Phi_n(q)$ has at least one negative coefficient if $\omega(n) \geq 2$.
        \item $\Phi_p(q) = \frac{1 - q^p}{1 - q} = 1 + q + q^2 + \cdots + q^{p - 1}$ if $p$ prime.
        \item $\Phi_n(q) = \Phi_{\rad(n)}(q^{n/\rad(n)})$.
        \item $\Phi_n(q)\Phi_{np}(q) = \Phi_n(q^p)$ if $p$ prime and $p \nmid n$.
        \item $\Phi_{np}(q) = \Phi_n(q^p)$ if $p$ prime and $p \mid n$.
    \end{enumerate}
\end{proposition}

\subsection{Cyclotomic generating functions (CGFs)}
In \cite{billey-swanson-cyclotomic}, the authors give several equivalent definitions for a CGF. We list the ones that will be useful to us.
\begin{theorem}\cite[Theorem 1]{billey-swanson-cyclotomic}
    Let $f(q)$ be a non-zero polynomial with non-negative integer coefficients in the variable $q$.
    Then, the following are equivalent:
    \begin{enumerate}[1.]
        \item $f(q)$ is a CGF.
        \item (Rational form.) $f(q)$ is of the form \eqref{equation:CGF-def}.
        \item (Complex form.) The complex roots of $f(q)$ are all either roots of unity or zero.
        \item (Cyclotomic form.) The polynomial~$f(q)$ can be written as a positive integer times a product of cyclotomic polynomials and factors of $q$.
    \end{enumerate}
\end{theorem}
The next lemma follows immediately from the definition in terms of cyclotomic polynomials.
\begin{lemma}\label{lemma:polynomial-iff-multiset-containment}
    Let $a_1, \ldots, a_m, b_1, \ldots, b_n \subset \mathbb{Z}_+$ with $m \geq n$. The following are equivalent:
    \begin{enumerate}[1.]
        \item The quotient
        \[
        \frac{\prod_{i = 1}^m 1 - q^{a_i}}{\prod_{j = 1}^n 1 - q^{b_i}} \in \mathbb{Z}[q].
        \]
        \item The inclusion \[
        \biguplus_{i = 1}^{n} D(b_i) \subseteq \biguplus_{i = 1}^{m} D(a_i)
        \]
        as multisets holds.
    \end{enumerate}
\end{lemma}

\subsection{The Frobenius coin problem}

Let $S = \{s_1, \ldots, s_n\} \subset \mathbb{Z}_+$ with $s_i < s_{i + 1}$ for all $i = 1, \ldots, n - 1$. The \emph{numerical semigroup} generated by $S$ is
\[
    \langle S \rangle_{\text{semigroup}} \coloneq \{ c_1s_1 + c_2s_2 + \cdots + c_ns_n: c_i \in \mathbb{Z}_{\geq 0}\}.
\]

Numerical semigroups are intimately related to the \emph{Frobenius coin problem}, which informally is about finding the monetary amounts that we can make with coins of given denominations. The classical formulation is actually to find the largest monetary amount that we \emph{cannot} make with these coins, but we will use the following stronger formulation.

\begin{problem}[Frobenius coin problem]
Assume $\gcd(S) = 1$. Determine $\mathbb{Z} \cap \langle S \rangle_{\text{semigroup}}$. In other words, the problem is to find all integers that are representable as a linear combination of elements of $S$ with non-negative integer coefficients.
\end{problem}

If $\gcd(S) = 1$, the quantity
\[
F(S) \coloneq \max \{i \in \mathbb{Z}: i \notin \langle S \rangle_{\text{semigroup}}\}
\]
is finite and is called the \emph{Frobenius number} of $S$. In general, giving a simple closed formula for $F(S)$ is only possible if $\# S = 2$. This being said, there are many known non-strict upper bounds for $F(S)$. One of the most famous such bounds is given by the following proposition.

\begin{proposition}[Selmer's bound \cite{selmer-linear}]\label{proposition:selmer-bound}
    If $\gcd(S) = 1$, then
    \[
    F(S) \leq 2s_{n - 1} \left \lfloor \frac{s_1}{n} \right \rfloor - s_1.
    \]
\end{proposition}

If $\gcd(S) \neq 1$, we will find it useful to consider the quantity $F(S/\gcd(S))$ instead.

For an excellent book about the Frobenius coin problem, see \cite{ramirez-alfonsin-diophantine}.

\subsection{Regular sequences}
Let $R \coloneq \mathbb{C}[x_1, \ldots, x_n]$ be the ring of polynomials in the variables $x_i$ with coefficients in $\mathbb{C}$. Let $f_1, \ldots, f_m \in R$ be a sequence of non-zero homogeneous polynomials. Then, the polynomials $f_1, \ldots, f_m$ form a \emph{regular sequence} of homogeneous elements if $f_i$ is not a zero-divisor in $R/(f_1 ,\ldots, f_{i-1})$ for all $i = 1, \ldots, m$. While regular sequences are defined for any commutative ring, since $R$ is \emph{Cohen-Macaulay}, the definition is equivalent to requiring that $\dim(V(f_1, \ldots, f_m)) = n - m$. For a very thorough treatment of these topics, see \cite{bruns-herzog-cohen-macaulay}.

\section{The algebro-geometric point of view}\label{section:algebraic-geometry}

In this section we give a necessary and sufficient combinatorial condition for the existence of regular sequences on graded polynomial rings with given sequences of degrees. This condition is given in terms of a family of Frobenius coin problems. To do so, we rely heavily on ideas in \cite[Theorem 8.7]{iano-fletcher-working} and \cite[Proposition 3.1]{pizzato-sano-tasin-effective}, who first introduced this as a criterion for quasismoothness of \emph{weighted complete intersections (WPIs)}, which is the corresponding object in projective algebraic geometry. We use \cite[Corollary 3.3]{stanley-hilbert} to translate this into a sufficient condition for the non-negativity of cyclotomic generating functions (CGFs). As a consequence, we solve Problem \ref{problem:problem-bs} and prove most cases of Conjectures \ref{conjecture:conjecture-stanton}, \ref{conjecture:conjecture-gk} and \ref{conjecture:conjecture-stanton-gk}.

We work over $\mathbb{C}$. For a sequence of degrees $b_1, \ldots, b_n \in \mathbb{Z}_+$, we set $\mathbb{P} \coloneq \mathbb{P}(b_1, \ldots, b_n)$ to be the \emph{weighted projective space} $\Proj~\mathbb{C}[x_1, \ldots, x_n] $ with $\deg(x_i) = b_i$. The space $\mathbb{P}$ is a projective toric variety of dimension $n - 1$.

For each degree $a \in \mathbb{Z}_+$, the reflexive rank-one sheaf $\mathcal{O}_{\mathbb{P}}(a)$ has a corresponding \emph{complete linear system}
\[
|\mathcal{O}_{\mathbb{P}}(a)| \coloneq \mathbb{P}(H^0(\mathbb{P}, \mathcal{O}_{\mathbb{P}}(a)))= \mathbb{P}(\{p \in \mathbb{C}[x_1, \ldots, x_n]: p \text{ homogeneous},~\deg(p) = a\}).
\]

The \emph{base locus} of $|\mathcal{O}_{\mathbb{P}}(a)|$ is the set of points in $\mathbb{P}$ where all elements of $|\mathcal{O}_{\mathbb{P}}(a)|$ vanish. 

For a subset of indices $I \subseteq \{1, \ldots, n\}$, we define the \emph{closed stratum}
\[
\Pi_I = \{x \in \mathbb{P}: i \notin I \implies x_i = 0 \} \cong \mathbb{P}(b_i: i \in I),
\]
and the \emph{open toric stratum}
\[
\Pi_I^{\circ} = \{x \in \mathbb{P}: i \in I \iff x_i \neq 0 \}.
\]

\begin{remark}
    Usually one works with \emph{well-formed} weighted projective spaces, meaning that
    \[
    \gcd(b_1, \ldots, b_{i - 1}, b_{i + 1}, \ldots, b_n) = 1
    \]
    for all $i = 1, \ldots, n$. Every weighted projective space is isomorphic to a well-formed one; the weights of the latter are calculated via a reduction algorithm. However, this reduction algorithm does \emph{not} generally preserve the sheaves $\mathcal{O}_{\mathbb{P}}(a)$. Indeed, the first step involves the natural isomorphism
    \[
    f: \mathbb{P} = \mathbb{P}(b_1, \ldots, b_n) \xrightarrow{\sim} \mathbb{P}' \coloneq \mathbb{P}(b_1/d, \ldots, b_n/d),
    \]
    where $d = \gcd(b_1, \ldots, b_n)$. Even though
    \[
    f^*\mathcal{O}_{\mathbb{P}'}(a) \cong \mathcal{O}_{\mathbb{P}}(da)
    \]
    for all $a \in \mathbb{Z}_+$, if $a < d$ and $a \in \langle b_1/d, \ldots, b_n/d \rangle_{\text{semigroup}}$, we have that $\dim(|\mathcal{O}_{\mathbb{P}'}(a)|) \geq 0$ but $|\mathcal{O}_{\mathbb{P}}(a)| = \emptyset$. Since the toric strata $\Pi_I$ may not be well-formed, we work with general weighted projective spaces and take special care in making sure that our complete linear systems are not empty.
\end{remark}

For another sequence of degrees $a_1, \ldots, a_m \in \mathbb{Z}_+$ with $m \leq n$, a \emph{weighted complete intersection (WCI)} $X_{a_1, \ldots, a_m} \subset \mathbb{P}$ is the vanishing set of a regular sequence of homogeneous elements $f_1, \ldots, f_m \in \mathbb{C}[x_1, \ldots, x_n]$ with $\deg(f_i) = a_i$. In other words,
\[
X_{a_1, \ldots, a_m} = \Proj~\frac{\mathbb{C}[x_1, \ldots, x_n]}{(f_1, \ldots, f_m)} = V(f_1, \ldots, f_m)
\]
is a projective variety of dimension $n - m - 1$. If $m = n$, then $X_{a_1, \ldots, a_n}$ is an \emph{empty} weighted complete intersection, but
\[
\frac{\mathbb{C}[x_1, \ldots, x_n]}{(f_1, \ldots, f_n)}
\]
is an Artinian ring and a finite-dimensional vector space. 

We will need the following technical lemma:

\begin{lemma} \label{lemma:technical-lemma}
    Let $b_1, \ldots, b_n \in \mathbb{Z}_+$ and $I \subseteq \{1, \ldots, n\}$. Consider the weighted projective space $\mathbb{P} \coloneq \mathbb{P}(b_1, \ldots, b_n)$ and the strata $\Pi_I$ and $\Pi_I^{\circ}$. Then:
    \begin{enumerate}[1.]
        \item Both $\Pi_I$ and $\Pi_I^{\circ}$ have dimension $\# I - 1$.
        \item $\overline{\Pi_I^{\circ}} = \Pi_I$ in the Zariski topology.
        \item Let $a \in \mathbb{Z}_+$. Then, a general element of the finite-dimensional complete linear system $|\mathcal{O}_{\mathbb{P}}(a)|$ restricts to a general element of $|\mathcal{O}_{\Pi_I}(a)|$.
        \item Let $a_1, \ldots, a_{\#I} \in \mathbb{Z}_+$, and assume that the linear systems $|\mathcal{O}_{\Pi_I}(a_i)|$ are base-point-free on $\Pi_I^{\circ}$ for all $i = 1, \ldots, \# I$. Then, cutting $\Pi_I^{\circ}$ by a choice of general elements $f_i \in |\mathcal{O}_{\Pi_I}(a_i)|$ for all $i = 1, \ldots, \# I$ results in the empty set.
    \end{enumerate}
\end{lemma}

\begin{proof}
    \begin{enumerate}[1.]
        \item This follows from the definitions.
        \item Idem.
        \item The map $r: H^0(\mathbb{P}, \mathcal{O}_{\mathbb{P}}(a)) \to H^0(\Pi_I, \mathcal{O}_{\Pi_I}(a))$ is surjective. Since both linear systems are finite-dimensional, $r$ is an open map.
        \item We give a proof by induction. Assume that for $i < \# I$ we have chosen general elements $f_j \in |\mathcal{O}_{\Pi_I}(a_j)|$ for all $j = 1, \ldots, i$ and the conclusion holds. In other words,
        \[
        X_i \coloneq \Pi_I^{\circ} \cap V(f_1, \ldots, f_i)
        \]
        has dimension $\# I - i - 1$. The complete linear system $|\mathcal{O}_{\Pi_I}(a_{i + 1})|$ is base-point-free on $X_{i, k} \subset X_i \subset \Pi_I^{\circ}$, where the $X_{i, k}$ are the (finitely many) irreducible components of $X_i$, so its base locus does not contain $X_{i, k}$ for any $k$. For each $k$, the set $V_k \subset |\mathcal{O}_{\Pi_I}(a_{i + 1})|$ of elements vanishing on all of $X_{i, k}$ is a linear subspace. Furthermore, $V_k$ cannot be all of $|\mathcal{O}_{\Pi_I}(a_{i + 1})|$, since otherwise its base locus would contain $X_{i, k}$. Therefore, $V_k$ is a closed subset of $|\mathcal{O}_{\Pi_I}(a_{i + 1})|$ and so is $V \coloneq \cup_{k} V_k$. Choose $f_{i + 1} \in |\mathcal{O}_{\Pi_I}(a_{i + 1})| \setminus V$, a general element. The polynomial $f_{i + 1}$ does not vanish on any irreducible component $X_{i, k}$. By Krull's principal ideal theorem, the dimension of $X_i \cap V(f_{i + 1})$ is $\# I - i - 2$. In particular, for $i = \# I - 1$,
        \[
        \dim(X_{\# I}) = \dim(X_{\# I - 1} \cap V(f_{\# I})) = -1,
        \]
        so $X_{\# I}$ is empty.
    \end{enumerate}
\end{proof}

We note that (Q1) in \cite[Proposition 3.1]{pizzato-sano-tasin-effective} can easily be translated into a combinatorial statement involving a family of Frobenius coin problems. However, while they characterise \emph{general} quasismooth WCIs in well-formed weighted projective spaces, we use similar ideas to characterise the existence of at least one WCI in general weighted projective spaces. It turns out that doing this is relatively straightforward, and that (Q2) in \cite[Proposition 3.1]{pizzato-sano-tasin-effective} is redundant if both degree sequences are of the same length (i.e., if $m = n$). This gives a solution to Problem \ref{problem:problem-bs}.

\begin{proposition} \label{proposition:regular-sequence-iff-wci-iff-frobenius-coin-problems}
    Let $a_1, \ldots, a_n, b_1, \ldots, b_n \in \mathbb{Z}_+$. Let $A = [a_i]_{i = 1, \ldots, n}$ and $B = [b_i]_{i = 1, \ldots, n}$ be the corresponding multisets. Then, the following are equivalent:
    \begin{enumerate}[1.]
        \item There exists a regular sequence of homogeneous elements $f_1, \ldots, f_n$ of maximal length with $\deg(f_i) = a_i$ in $\mathbb{C}[x_1, \ldots, x_n]$ with $\deg(x_i) = b_i$. In other words, there exists an empty weighted complete intersection $X_{a_1, \ldots, a_n} \subset \mathbb{P}(b_1, \ldots, b_n)$. 
        \item For every subset of indices $I \subseteq \{1, \ldots, n\}$, the inequality~$\# (A~\overline{\cap}~\langle \toset(B_I) \rangle_{\text{semigroup}}) \geq \# I$ holds.
    \end{enumerate}
\end{proposition}

\begin{proof}
    \begin{itemize}
        \item $1 \implies 2$: Let $f_1, \ldots, f_n$ be a regular sequence of homogeneous elements with $\deg(f_i) = a_i$. Assume 3 is not true for some $I \subseteq \{1, \ldots, n\}$. If $f_i$ is such that
        \[
        \deg(f_i) = a_i \notin A~\overline{\cap}~\langle \toset(B_I) \rangle_{\text{semigroup}},
        \]
        then all of its terms involve a coordinate $x_i$ with $b_i \notin B_I$. Therefore, the closed stratum $\Pi_I \subset V(f_i)$.
        Now, $\dim(\Pi_I) = \# I - 1$ and
        \[
        0 \leq \# I - \# (A~\overline{\cap}~\langle \toset(B_I) \rangle_{\text{semigroup}}) - 1\leq \dim(\Pi_I \cap V(f_1, \ldots, f_n)).
        \]
        Therefore, $V(f_1, \ldots, f_n)$ contains at least one point in $\mathbb{P}$, so the $f_1, \ldots, f_n$ do not form a regular sequence.
                
        \item $2 \implies 1$: Let $f_1, \ldots, f_n$ be general homogeneous elements with $\deg(f_i) = a_i$. They exist because all of the $a_i$ are representable with the $b_i$ (just apply the hypothesis to $I = \{1, \ldots, n\}$). We show that the $f_1, \ldots, f_n$ form a regular sequence.
        
        Indeed, let $p \in \mathbb{P}$. We show that $p \notin V(f_1, \ldots, f_n)$. Let $I \subseteq \{1, \ldots, n\}$. For all $a_i \in A~\overline{\cap}~\langle \toset(B_I) \rangle_{\text{semigroup}}$, there is a monomial of degree $a_i$ involving only the variables $x_i$ with $i \in I$, so this monomial does not vanish anywhere on $\Pi_I^{\circ}$. Therefore, the complete linear system $|\mathcal{O}_{\Pi_I}(a_i)|$ is non-empty and base-point-free on $\Pi_I^{\circ}$. By Lemma \ref{lemma:technical-lemma}, the restriction of $f_i$ to $\Pi_I$ is also general. Also by the lemma, cutting $\Pi_I^{\circ}$ by the $f_i$ with $a_i \in A~\overline{\cap}~\langle \toset(B_I) \rangle_{\text{semigroup}}$ results in the empty set, since there are at least $\# I$ equations. The point $p$ is contained in exactly one $\Pi_I^{\circ}$, so the claim follows.
    \end{itemize}
\end{proof}

\begin{remark}
    We can also check if $1 \implies 2$ is consistent with \cite[Proposition 3.1]{pizzato-sano-tasin-effective}. If we do not allow linear cones and we assume that $\mathbb{P}$ is well-formed, 1 is exactly the hypothesis in \cite[Proposition 3.1]{pizzato-sano-tasin-effective} applied to $X \coloneq X_{a_1, \ldots, a_{n + 1}} \subset \mathbb{P} \coloneq \mathbb{P}(b_0, \ldots, b_n)$ after relabelling indices ($X$ is of Krull dimension $-1$, so empty and trivially smooth). We conclude that either (Q1) or (Q2) holds for all $I \subseteq \{0, \ldots, n\}$. It is easy to see that that having (Q1) for all $I$ is equivalent to 3, so we show that (Q2) leads to a contradiction for all $I$.

    Indeed, assume (Q2) is true for some $I$, and use the notation in \cite[Proposition 3.1]{pizzato-sano-tasin-effective}. The integers $e_{\mu, j}$ are chosen from a set of cardinality $n + 1 - k$. For $J = \{l + 1, \ldots, c\}$ we have that
    \begin{align*}
        &\# \{e_{\mu, j}: \mu \in \{1, \ldots, k - l\}, j \in J\} \geq k - l + \# J - 1 = k - 2l + c - 1 \\
        \implies &n + 1 - k \geq k - 2l + c - 1 \\
        \implies &c \leq n - 2k + 2l + 2 = n - 2(k - l) + 2 \leq n.
    \end{align*}
    But $c = n + 1$, a contradiction.
\end{remark}

The following proposition follows from an argument analogous to that in the proof of Proposition \ref{proposition:regular-sequence-iff-wci-iff-frobenius-coin-problems}. Again, the statement of the condition is given in \cite[Proposition 3.1]{pizzato-sano-tasin-effective}.

\begin{proposition}
    Let $a_1, \ldots, a_m, b_1, \ldots, b_n \in \mathbb{Z}_+$ with $m < n$. Let $A = [a_i]_{i = 1, \ldots, n}, B = [b_i]_{i = 1, \ldots, n}$ be the corresponding multisets. Assume that:
    \begin{enumerate}[i.]
        \item $a_i \neq b_j$ for all $i, j$. \emph{(No linear cones.)}
        \item $\gcd(b_1, \ldots, b_{i - 1}, b_{i + 1} \ldots b_n) = 1$ for all $i$. \emph{(Well-formed weighted projective space.)}
    \end{enumerate}
    Then, $1 \impliedby 2$:
    \begin{enumerate}[1.]
        \item There exists a regular sequence of homogeneous elements $f_1, \ldots, f_m$ with $\deg(f_i) = a_i$ in $\mathbb{C}[x_1, \ldots, x_n]$ with $\deg(x_i) = b_i$. In other words, there exists a weighted complete intersection $X_{a_1, \ldots, a_m} \subset \mathbb{P}(b_1, \ldots, b_n)$.
        \item For every subset of indices $I \subseteq \{1, \ldots, n\}$, the inequality $\# (A~\overline{\cap}~\langle \toset(B_I) \rangle_{\text{semigroup}}) \geq \min(m, \# I)$ holds.
    \end{enumerate}
\end{proposition}

The following proposition follows from \cite[Corollary 3.3]{stanley-hilbert}.

\begin{proposition} 
    Let $a_1, \ldots, a_m, b_1, \ldots, b_n \in \mathbb{Z}_+$ with $m \leq n$. Then, $1 \implies 2$:
    \begin{enumerate}[1.]
        \item There exists a regular sequence of homogeneous elements $f_1, \ldots, f_m$ with $\deg(f_i) = a_i$ in $\mathbb{C}[x_1, \ldots, x_n]$ with $\deg(x_i) = b_i$.
        \item The quotient
        \[
        \frac{\prod_{i = 1}^m 1 - q^{a_i}}{\prod_{j = 1}^n 1 - q^{b_i}} \in \mathbb{Z}_{\geq 0}[\![ q ]\!].
        \]
    \end{enumerate}
    In particular, if $m = n$, the quotient
    \[
    \prod_{i = 1}^n \frac{1 - q^{a_i}}{1-q^{b_i}} \in \mathbb{Z}_{\geq 0}[q].
    \]
\end{proposition}

\begin{remark}
    \emph{Why don't we also have $2 \implies 1$?} Assume 2. It is easy to see that all the $a_i$ must be representable with the $b_i$. Therefore, there exist non-zero homogeneous elements $f_1, \ldots, f_m$ with $\deg(f_i) = a_i$ in $\mathbb{C}[x_1, \ldots, x_n]$ with $\deg(x_i) = b_i$. By [Stanley, Corollary 3.2],
    \[
    \frac{\prod_{i = 1}^m 1 - q^{a_i}}{\prod_{j = 1}^n 1 - q^{b_i}} \leq \Hilb \left ( \frac{\mathbb{C}[x_1, \ldots, x_n]}{(f_1, \ldots, f_m)} \right ),
    \]
    where equality is equivalent to the $f_i$ being a regular sequence, so we cannot conclude. Indeed, \cite[Remark 32]{billey-swanson-cyclotomic} gives a counterexample:
    \[
    \frac{(1 - q^2)(1 - q^3)^2(1 - q^8)(1 - q^{12})}{(1 - q)^2(1 - q^4)^2(1-q^6)} \in \mathbb{Z}_{\geq 0}[q],
    \]
    and it is easy to see (directly, or by using Proposition \ref{proposition:regular-sequence-iff-wci-iff-frobenius-coin-problems}) that there does not exist a weighted complete intersection $X_{2,3,3,8,12} \subset \mathbb{P}(1,1,4,4,6)$.
\end{remark}

\begin{remark}
    If the quotient
    \[
    \frac{\prod_{i = 1}^m 1 - q^{a_i}}{\prod_{j = 1}^n 1 - q^{b_i}} \in \mathbb{Z}[q]
    \]
    but has at least one negative coefficient, a general choice of homogeneous $f_i$ may not be a regular sequence. Indeed,
    \[
    \frac{(1-q^3)(1-q^5)(1-q^{14})}{(1-q^2)(1-q^3)(1-q^7)} \in \mathbb{Z}[q]
    \]
    and has at least one negative coefficient. Since $3, 5, 14 \in \langle 2, 3, 7 \rangle_{\text{semigroup}}$, there exist general homogeneous $f_1, f_2, f_3 \in \mathbb{C}[x_1, x_2, x_3]$ with
    \[
    \deg (x_1) = 2, \quad \deg (x_2) = 3, \quad \deg (x_3) = 7
    \]
    and
    \[
    \deg (f_1) = 3, \quad \deg (f_2) = 5, \quad \deg (f_3) = 14,
    \]
    but they cannot form a regular sequence.
\end{remark}

We may now state and prove our main theorem.

\begin{theorem} \label{theorem:main-theorem}
    Let $a_1, \ldots, a_n, b_1, \ldots, b_n \in \mathbb{Z}_+$ with $a_i \leq a_{i + 1}$ and $b_i \leq b_{i + 1}$ for all $i = 1, \ldots, n - 1$. Let $A = [a_i]_{i = 1, \ldots, n}, B = [b_i]_{i = 1, \ldots, n}$ be the corresponding multisets. Assume that
    \[
    \prod_{i = 1}^n \frac{1 - q^{a_i}}{1 - q^{b_i}} \in \mathbb{Z}[q].
    \]
    Assume further that
    \[
    2b_n\left \lfloor \frac{b_{n - 1}}{2} \right \rfloor - b_1 < a_1.
    \]
    Then, for every subset of indices $I \subseteq \{1, \ldots, n\}$, the inequality $\# (A~\overline{\cap}~\langle \toset(B_I) \rangle_{\text{semigroup}}) \geq \# I$ holds. In particular,
    \[
    \prod_{i = 1}^n \frac{1 - q^{a_i}}{1 - q^{b_i}} \in \mathbb{Z}_{\geq 0}[q].
    \]
\end{theorem}

\begin{proof}
    Let $I \subseteq \{1, \ldots, n\}$. Two cases:
    \begin{enumerate}[1.]
        \item $\#(A~\overline{\cap}~\mathbb{Z}\toset(B_I)) \geq \# I$: This case is trivial since $A~\overline{\cap}~\mathbb{Z}\toset(B_I) \subseteq A~\overline{\cap}~\langle \toset(B_I) \rangle_{\text{semigroup}}$.
        \item $\#(A~\overline{\cap}~\mathbb{Z}\toset(B_I)) < \# I$: By Lemma \ref{lemma:polynomial-iff-multiset-containment}, we have that
        \[
        B_I \subseteq B \subset \biguplus_{i = 1}^n D(b_i) \subseteq \biguplus_{i = 1}^{n} D(a_i),
        \]
        so $g \coloneq \gcd(\toset(B_I))$ is in $\biguplus_{i = 1}^n D(a_i)$ with multiplicity at least $\# B_I$ and so $\#(A~\overline{\cap}~\mathbb{Z}g) \geq \# B_I = \# I$. Therefore, it suffices to show that $A~\overline{\cap}~\mathbb{Z}g \subseteq A~\overline{\cap}~\langle \toset(B_I) \rangle_{\text{semigroup}}$.

        By Selmer's bound \cite{selmer-linear} (here, Proposition \ref{proposition:selmer-bound}) to the Frobenius number $F(B_I/g)$,
        \[
        \max \{i \in \mathbb{Z}g: i \notin \langle \toset(B_I) \rangle_{\text{semigroup}} \} \leq 2 \max(\toset(B_I)) \left \lfloor \frac{\min(\toset(B_I))}{g\# \toset(B_I)} \right \rfloor - \min(\toset(B_I)).
        \]
        Therefore, it is enough that
        \[
        2b_n\left \lfloor \frac{b_{n - 1}}{2} \right \rfloor - b_1 < a_1,
        \]
        since, without loss of generality, $\# \toset(B_I) \geq 2$.
    \end{enumerate}
\end{proof}
\begin{corollary} \label{corollary:conjecture-gk-upper-bound}
    Conjecture \ref{conjecture:conjecture-gk} is true if
    \[
    2k \left \lfloor \frac{k - 1}{2} \right \rfloor + k - l < n.
    \]
    In particular, it is true if $n \geq k^2$.
    Hence, Conjecture \ref{conjecture:conjecture-gk} is true except perhaps for finitely many choices of $l, n$ for each choice of $k$.
\end{corollary}

\begin{corollary} \label{corollary:conjecture-stanton-upper-bound}
    Conjectures \ref{conjecture:conjecture-stanton} and \ref{conjecture:conjecture-stanton-gk} are true if
    \[
    2n \left \lfloor \frac{n}{2} \right \rfloor - 2 < m.
    \]
    In particular, they are true if $m \geq n^2$. 
    Hence, they are true except perhaps for finitely many choices of $m$ for each choice of $n$.
\end{corollary}

\begin{remark}
    Since Theorem \ref{theorem:main-theorem} is quite general, it is very likely that Proposition \ref{proposition:regular-sequence-iff-wci-iff-frobenius-coin-problems} can be used to prove even more cases of the aforementioned conjectures by using specific properties of the multisets $A$ and $B$. For example, Conjecture \ref{conjecture:conjecture-gk} involves \emph{intervals}, which, along with the polinomiality, may force the elements of $A$ to be much larger than the elements of $B$.
\end{remark}

\section{The lattice-theoretic and analytic points of view}
In this section we give more sufficient conditions for the non-negativity of cyclotomic generating functions (CGFs). This time, the conditions arise from structural properties of the \emph{division lattice} $(\mathbb{Z}_{\geq 0}, \mid)$. We will also see that such conditions cannot be necessary, since CGFs are often non-negative for \emph{analytic} reasons. In particular, we will briefly discuss Poincaré multipliers and asymptotic normality.

By abuse of notation, we will sometimes denote $(\mathbb{Z}_{\geq 0}, \mid)$ by $\mathbb{Z}_{\geq 0}$. It is well-known that it is a distributive lattice with the meet operation being the GCD and the join operation being the LCM. Furthermore, $\prec$ will denote a \emph{covering} relation.

\begin{proposition}\label{proposition:quotient-ai-bi}
    Let $a_1, \ldots, a_m, b_1, \ldots, b_n \in \mathbb{Z}_+$ with $m \geq n$. Let $A = [a_i]_{i = 1, \ldots, m}, B = [b_i]_{i = 1, \ldots, n}$ be the corresponding multisets. Assume that the quotient
    \begin{equation}\label{equation:quotient-ai-bi}
        \frac{\prod_{i = 1}^m 1 -q^{a_i}}{\prod_{i = 1}^n 1 - q^{b_j}} \in \mathbb{Z}[q].
    \end{equation}
    The following are sufficient conditions for \eqref{equation:quotient-ai-bi} to have non-negative coefficients:
    \begin{enumerate}[1.]
        \item The multiset
        \[
        \Delta \coloneq \biguplus_{i = 1}^m D(a_i) \setminus \biguplus_{j = 1}^n D(b_j)
        \]
        is a disjoint multiset union of sets of the form $p^k \lor I$, where $p$ prime, $k \in \mathbb{Z}_+$, and $I$ is a principal lower order ideal of $\mathbb{Z}_{\geq 0}$ with $p \notin I$ (equivalently, $I = D(d)$ for some $d$ with $p \nmid d$).
        \item The multiset $\Delta$ is a disjoint multiset union of sets $D_k$ such that, for each $k$, either
        \begin{enumerate}[a.]
            \item $\omega(D_k) \subset \{0, 1\}$, or
            \item $D_k = \{d, d'\}$, where $d \prec d'$ and $\omega(d) \leq 1$.
        \end{enumerate}
        \item There exists a bijection $f: \{1, \ldots, n\} \to \{1, \ldots, m\}$ such that $D(b_j) \subseteq D(a_{f(j)})$ (equivalently, $b_j | a_{f(j)}$) for all $j \in \{1, \ldots, n\}$. In this case, the corresponding local quotients are non-negative and flat.
    \end{enumerate}

    The following are sufficient conditions for the local quotients in \eqref{equation:quotient-ai-bi} to be flat:
    \begin{enumerate}[i.]
        \item The multiset $\Delta$ is a disjoint multiset union of sets $D_k$ such that, for each $k$, either
        \begin{enumerate}[a.]
            \item $\omega(D_k) \subset \{0, 1, 2\}$, or
            \item $D_k = \{d, d'\}$, where $d \prec d'$ and $\omega(d) \leq 2$.
        \end{enumerate}
    \end{enumerate}

    All of these conditions also have a corresponding ``local'' version which is generally more useful in examples, because then it is possible to use different criteria for different local quotients.
\end{proposition}

\begin{proof}
    \begin{enumerate}[1.]
        \item Let $\mathcal{I}\mathcal{P}\mathcal{K}$ be the corresponding set of triplets. The quotient \eqref{equation:quotient-ai-bi} is equal to
        \[
        \prod_{(I, p, k) \in \mathcal{I}\mathcal{P}\mathcal{K}} \prod_{n \in p^k \lor I } \Phi_n(q).
        \]
        For each $(I, p, k) \in \mathcal{I}\mathcal{P}\mathcal{K}$, we have that
        \[
        \prod_{n \in p^k \lor I} \Phi_n(q) \prod_{n \mid p^{k - 1}d} \Phi_n(q) = \prod_{\substack{n = p^k d' \\ d' \mid d}} \Phi_n(q)\prod_{n \mid p^{k - 1}d} \Phi_{n}(q) = \prod_{n \mid p^k d} \Phi_n(q),
        \]
        therefore
        \[
        \prod_{n \in p^k \lor I} \Phi_n(q) = \frac{1-q^{p^k d}}{1-q^{p^{k - 1} d}}.
        \]
        Substituting $q' = q^{p^{k - 1} d}$, we see that this is a geometric sum.
        \item Let $\mathcal{K}$ be the corresponding set of indices. As in 1, we decompose the quotient \eqref{equation:quotient-ai-bi} into a product indexed by $k \in \mathcal{K}$. If $D_k$ is of type a (respectively, b), the conclusion follows from Proposition \ref{proposition:cyclotomic-polynomials} 3 and 4 (respectively, 3 and 5). 
        \item The quotient \eqref{equation:quotient-ai-bi} is equal to
        \[
        \prod_{j \in \{1, \ldots, n\}} \frac{1 - q^{a_{f(j)}}}{1 - q^{b_j}}.
        \]
        For each $j = 1, \ldots, n$, substituting $q' = q^{b_j}$, we see that this is a geometric sum.
    \end{enumerate}
    \begin{enumerate}[i.]
        \item As in 1, we decompose the quotient \eqref{equation:quotient-ai-bi} into a product indexed by $k \in \mathcal{K}$. If $D_k$ is of type a (respectively, b), the conclusion follows from Proposition \ref{proposition:cyclotomic-polynomials} 1 and 4 (respectively, 1 and 5).
    \end{enumerate}
\end{proof}

\begin{example}
    Even if $m = n$, the quotient \eqref{equation:quotient-ai-bi} may be a polynomial with at least one negative coefficient. For example,
    \[
    \frac{(1-q^{105})(1-q^3)(1-q^5)(1-q^7)}{(1-q^{35})(1-q^{21})(1-q^{15})(1-q)} = \Phi_{105}(q)
    \]
    which famously has a negative coefficient of $-2$.
\end{example}

\begin{remark}
    In Proposition \ref{proposition:quotient-ai-bi} 3, there may exist a well-defined function $f: \{1, \ldots, n\} \to \{1, \ldots, m\}$ such that $D(b_j) \subseteq D(a_{f(j)})$ for all $j \in \{1, \ldots, n\}$, but $f$ may not be a bijection, even if $m = n$. Indeed, in Example \ref{example:two-bi-one-ai}, $1$ and $3$ must both map to $1$ and $2$ must map to $2$ (because $b_1, b_3 \mid a_1$ and $b_2 \mid a_2$), so $3$ must have empty preimage (because $b_i \nmid a_3$ for all $i = 1, 2, 3$).
    
    In general, however, if the quotient \eqref{equation:quotient-ai-bi} is a polynomial, there always exists a well-defined function $f$ with the property mentioned above, simply because, for all $j \in \{1, \ldots, n\}$, the cyclotomic polynomial $\Phi_{b_j}$ is a factor of $1 - q^{a_i}$ for some $i \in \{1, \ldots, m\}$, which implies that $b_j \mid a_i$.
\end{remark}

\begin{example}
    Consider the polynomial quotient:
    \[
    \frac{(1-q^{60})(1-q^{66})(1-q^{72})}{(1-q^{10})(1-q^{11})(1-q^{12})}.
    \]
    The second local quotient is equal to:
    \[
    \Phi_{2} \Phi_{3} \Phi_{2 \cdot 3} \Phi_{2 \cdot 11} \Phi_{3 \cdot 11} \Phi_{2 \cdot 3 \cdot 11}.
    \]
    We could apply Proposition \ref{proposition:quotient-ai-bi} 1 to deduce local non-negativity, with the triplets
    \[
    (D(11), 2, 1) \quad \text{ and } \quad (D(2 \cdot 11), 3, 1), 
    \]
    and similarly for the other local quotients.

    We could also apply Proposition \ref{proposition:quotient-ai-bi} 3 to deduce global non-negativity, since
    \[
    10 \mid 60, \quad  11 \mid 66, \quad 12 \mid 72.
    \]
\end{example}

If at least one of the conditions in Proposition \ref{proposition:quotient-ai-bi} were generally true for the sequences
\[
(a_i)_{i = 1, \ldots, k - l} = (l + 1, l + 2, \ldots, k) \quad \text{ and } \quad (b_j)_{i = 1, \ldots, k - l} = (n - k + 1, n - k + 2, n - k),
\]
Conjecture \ref{conjecture:conjecture-gk} would follow. Unfortunately, this is not the case, as the following example shows:

\begin{example} \label{example:two-bi-one-ai}
    Consider the polynomial quotient:
    \[
    \frac{\binom{271}{17}_q}{\binom{271}{14}_q} = \frac{(1-q^{255})(1-q^{256})(1-q^{257})}{(1-q^{15})(1-q^{16})(1-q^{17})} = \Phi_{3 \cdot 17} \Phi_{5 \cdot 17} \Phi_{3 \cdot 5 \cdot 17} \Phi_{2^5} \Phi_{2^6} \Phi_{2^7} \Phi_{2^8} \Phi_{257},
    \]
    which, as can be seen by using e.g. SageMath \cite{sagemath}, has non-negative coefficients (!) This is remarkable, since $\Phi_{3 \cdot 17} \Phi_{5 \cdot 17} \Phi_{3 \cdot 5 \cdot 17}$ has negative coefficients (it is also flat by Proposition \ref{proposition:quotient-ai-bi} i). In other words, the non-negativity does \emph{not} come from a geometric sum or a low-omega argument such as those used in the proof of Proposition \ref{proposition:quotient-ai-bi}, but from a genuine interaction between the different local quotients. SageMath also tells us that $\Phi_{2^5} \Phi_{2^6} \Phi_{2^7} \Phi_{2^8} \Phi_{257}$ has a normal-like coefficient distribution.
\end{example}

\begin{theorem} \cite{poincare-equations} Let $f \in \mathbb{R}[x]$ be a polynomial. The following are equivalent:
    \begin{enumerate}[1.]
        \item $f(x) > 0$ for all $x \geq 0$.
        \item There exists $g \in \mathbb{R}[x]$ monic such that $fg$ has only non-negative coefficients. The polynomial~$g$ is called a \emph{Poincaré multiplier.}
    \end{enumerate}
\end{theorem}

\begin{theorem} \cite{polya-positive} \label{theorem:polya} Let $f \in \mathbb{R}[x]$ be a polynomial such that $f(x) > 0$ for all $x \geq 0$. Then, there exists $k \in \mathbb{Z}_+$ such that $f(x)(1 + x)^k$ has only non-negative coefficients.
\end{theorem}

\begin{corollary}
    Theorem \ref{theorem:polya} applies to finite products of cyclotomic polynomials $\Phi_{n_i}$ with $n_i \neq 1$ for all $i$.
\end{corollary}

\begin{proof}
    If $n_i \neq 1$, it is well-known that $\Phi_{n_i}(0) = 1$. Also, $\Phi_{n_i}$ has all of its roots in the unit circle. Since $n_i \neq 1$, $1$ is not a root. Therefore, $\Phi_{n_i}(x) > 0$ for all $x \geq 0$.
\end{proof}

Note that $(1 + x)^k$ also has a normal-like coefficient distribution. Therefore, it seems that multiplying by polynomials with normal-like coefficient distributions makes the product have non-negative coefficients. This is the \emph{analytic} reason for the non-negativity of certain CGFs. In this direction, results about asymptotic normality of coefficient distributions of products of cyclotomic polynomials, such as those in \cite{billey-swanson-cyclotomic}, could help find more such necessary or sufficient criteria.

\section*{Acknowledgements}
We would like to thank Balázs Szendrői for his support and many useful comments. We would also like to thank Livia Campo, Pedro Montero and Markus Reibnegger for insightful conversations. We used the model Gemini 3 Pro by Google as a tool for informal discussion of ideas and generation of experimental evidence, and SageMath \cite{sagemath} for various computations.

\bibliographystyle{amsalpha}
\bibliography{main}

\end{document}